\documentclass[reqno]{amsart}
\usepackage{amssymb}
\usepackage[mathscr]{euscript}
\usepackage{amsmath}
\makeatletter
\@addtoreset{equation}{section}
\makeatother

\renewcommand\thefigure{\thesection.\@arabic\c@figure}
\renewcommand\thetable{\thesection.\@arabic\c@table}

\newtheorem{theorem}{Theorem}[section]
\newtheorem{lemma}[theorem]{Lemma}
\newtheorem{proposition}[theorem]{Proposition}
\newtheorem{corollary}[theorem]{Corollary}

\newcommand{\mc}[1]{{\mathcal #1}}
\newcommand{\ms}[1]{{\mathscr #1}}
\newcommand{\mf}[1]{{\mathfrak #1}}
\newcommand{\mb}[1]{{\mathbf #1}}
\newcommand{\bb}[1]{{\mathbb #1}}
\newcommand{\bs}[1]{{\boldsymbol #1}}

\newcommand{\<}{\langle}
\renewcommand{\>}{\rangle}

\renewcommand{\Cap}{{\rm cap}}

\begin{document}

\title[Metastability of continuous time Markov
  chains]{Tunneling and Metastability of continuous time Markov
  chains II, the nonreversible case}

\author{J. Beltr\'an, C. Landim}

\address{\noindent IMCA, Calle los Bi\'ologos 245, Urb. San C\'esar
  Primera Etapa, Lima 12, Per\'u and PUCP, Av. Universitaria cdra. 18,
  San Miguel, Ap. 1761, Lima 100, Per\'u. 
\newline e-mail: \rm
  \texttt{johel@impa.br} }

\address{\noindent IMPA, Estrada Dona Castorina 110, CEP 22460 Rio de
  Janeiro, Brasil and CNRS UMR 6085, Universit\'e de Rouen, Avenue de
  l'Universit\'e, BP.12, Technop\^ole du Madril\-let, F76801
  Saint-\'Etienne-du-Rouvray, France.  \newline e-mail: \rm
  \texttt{landim@impa.br} }

\keywords{Metastability, Tunneling, Markov processes.} 

\begin{abstract}
  We proposed in \cite{bl2} a new approach to prove the metastable
  behavior of reversible dynamics based on potential theory and local
  ergodicity. In this article we extend this theory to nonreversible
  dynamics based on the Dirichlet principle proved in \cite{gl2}.
\end{abstract}

\maketitle

\section{Introduction}

Metastability is a relevant dynamical phenomenon in non-equilibrium
statistical mechanics which occurs in the vicinities of first order
phase transitions \cite{ov}, \cite{g}.  In the sequel of \cite{lp,
  cgov, begk1}, we proposed in \cite{bl2} a new definition of
metastability for continuous time Markov chains which visit points,
and we presented in the reversible case three simple conditions which
ensure the metastable behavior of a chain. All three conditions are
formulated in terms of the stationary measure and of the capacity
between the metastable sets.

This method has been applied successfully in \cite{bl3} to prove the
metastable behavior of the condensate in reversible zero range
processes evolving on finite sets, in \cite{bl4, bl5} to examine the
metastability of reversible Markov chains evolving on fixed finite
sets, and in \cite{jlt1, jlt2} to investigate the scaling limits of
trap models on random graphs.

The proof of the metastable behavior of a reversible Markov chain
presented in \cite{bl2} relies essentially on two ingredients. On the
one hand, the potential theory of reversible Markov chains
\cite{begk2, g} permits to express the capacity between two sets in
terms of a variational problem, the so-called Dirichlet principle, and
provides simple formulas for the expectations of time integrals,
carried over time intervals ending at the hitting time of a set, in
terms of the equilibrium potentials. On the other hand, the local
ergodicity allows to replace a time integral of a function by the time
integral of the average of the function over the metastable sets.

We generalized recently in \cite{gl2} the potential theory of
reversible Markov chains to the non-reversible context by defining a
capacity between two sets and by proving a Dirichlet principle which
involves, in contrast with the reversible case, a double variational
formula.

In this article, we extend the theory elaborated in \cite{bl2} to
nonreversible dynamics aiming to examine the metastable behavior of
the condensate in asymmetric zero range processes evolving on a
one-dimensional finite torus \cite{gl3}. The main result, stated in
Theorem \ref{teo1r} below, asserts that the metastable behavior of a
sequence of continuous time Markov chains follows from the same three
conditions formulated in the reversible case.

One of the conditions, assumption ({\bf H0}) below, requires the
average jump rates of the trace process to converge. While in the
reversible case the average rates can be expressed in terms of
capacities, \cite[Remark 2.9]{bl2}, in the nonreversible case such
formula is not available.  We may, however, formulate a double
variational problem whose optimal solution is equal to the average
rate. This is the content of Proposition \ref{s05}.

It may seem that the characterization of the average jump rate
provided by Proposition \ref{s05} is deprived of interest and
impossible to use in concrete examples. This is not the case as shown
in \cite{gl3} where we apply this result to obtain the asymptotic
value of the average jump rates. This is done in two steps. We first
compute the limit of \eqref{09} as $N\uparrow\infty$ and derive the
optimal value for $f$ on the set $\ms E^x_N$. We then show that if the
supremum in \eqref{09} is carried over functions $h$ which are not
equal to the optimal value for $f$ on $\ms E^x_N$, the limit of
\eqref{09} is strictly smaller than the one obtained before.

The last main result of the article, Proposition \ref{s06}, asserts
that the adjoint of a Markov chain which exhibits a metastable
behavior is also metastable and its asymptotic dynamics is the adjoint
of the asymptotic dynamics of the original process.

\section{Notation and Results}

Fix a sequence $(E_N: N\ge 1)$ of countable state spaces. The elements
of $E_N$ are denoted by the Greek letters $\eta$, $\xi$. A sequence of
states $\bs \eta=(\eta^N\in E_N : N\ge 1)$ is said to be a point in a
sequence ${\ms A}$ of subsets of $E_N$, ${\ms A} = (A_N\subseteq E_N :
N\ge 1)$, if $\eta^N$ belongs to $A_N$ for every $N\ge 1$.

For each $N\ge 1$ consider a matrix $R_N : E_N \times E_N \to \bb R$
such that $R_N(\eta, \xi) \ge 0$ for $\eta \not = \xi$, $-\infty < R_N
(\eta, \eta)\le 0$ and $\sum_{\xi\in E_N} R_N(\eta,\xi)=0$ for all
$\eta\in E_N$. Denote by $L_N$ the generator which acts on bounded
functions $f:E_N\to \bb R$ as
\begin{equation}
\label{c01}
(L_Nf) (\eta) \,=\, \sum_{\xi\in E_N} R_N(\eta,\xi)
\, \big\{f(\xi)-f(\eta)\big\}\;.
\end{equation}

Let $\{\eta^N_t : t\ge 0\}$ be the {\sl minimal} right-continuous
Markov process associated to the generator $L_N$. It is well known
that $\{\eta^N_t : t\ge 0\}$ is a strong Markov process with respect
to the filtration $\{\mc F^N_t : t\ge 0\}$ given by $\mc F^N_t =
\sigma (\eta^N_s : s\le t)$ \cite{f, n}. We assume throughout this
section that for each $N\ge 1$ the process $\{\eta^N_t : t\ge 0\}$ is
\emph{irreducible} and \emph{positive recurrent}.

Denote by $\mu^N$ the unique stationary probability measure and let
$\lambda^N(\eta) = -R(\eta,\eta)$, $\eta\in E_N$, be the holding rates
of the process $\eta^N_t$. We assume that the holding rates are
integrable:  
\begin{equation}
\label{re1}
\sum_{\eta \in E_N} \lambda_N(\eta)\mu_N(\eta) \;<\; \infty 
\;,\quad \forall N\ge 1\;.
\end{equation}

Denote by $D(\bb R_+,E_N)$ the space of right-continuous trajectories
with left limits endowed with the Skorohod topology. Let
$\mb P^N_{\eta}$, $\eta\in E_N$, be the probability measure on $D(\bb
R_+,E_N)$ induced by the Markov process $\{\eta^N_t : t\ge 0\}$
starting from $\eta$. Expectation with respect to $\mb P^N_{\eta}$ is
denoted by $\mb E^N_{\eta}$ and we frequently omit the index $N$ in
$\mb P^N_{\eta}$, $\mb E^N_\eta$.

Denote by $H_A$ (resp. $H^+_A$), $A\subseteq E_N$, the hitting time of
(resp. return time to) the set $A$:
\begin{equation*}
\begin{split}
& H_{A} \;:=\;\inf\{ t> 0 : \eta^N_t\in A \}\;, \\
& \quad H^+_A \,:=\, \inf\{ t>0 : \eta^N_t\in A \,,\,
\eta^N_s\not= \eta^N_0 \;\;\textrm{for some $0< s < t$}\}\,.
\end{split}
\end{equation*}

Denote by $L^2(\mu_N)$ the space of square integrable functions
$f:E_N\to\bb R$, $\sum_{\eta\in E_N} f(\eta)^2 \mu_N(\eta)<\infty$,
endowed with the usual scalar product. Let $L^*$ be the adjoint of $L$
in $L^2(\mu_N)$. The operator $L^*$ corresponds to the generator of a
Markov process, denoted by $\{\eta^{*,N}_t : t\ge 0\}$, whose rates are
represented by $R^*_N$. Denote by $\mb P^* = \mb P^{*,N}_{\eta}$,
$\eta\in E_N$, the probability measure on $D(\bb R_+,E_N)$ induced by
the adjoint Markov process starting from $\eta$. Expectation with
respect to $\mb P^{*,N}_{\eta}$ is represented by $\mb E^{*,N}_{\eta}$

Fix a finite number of disjoint subsets $\ms E^1_N, \dots, \ms
E^\kappa_N$, $\kappa\ge 2$, of $E_N$: $\ms E^x_N\cap \ms
E^y_N=\varnothing$, $x\neq y$. Let $S=\{1, \dots, \kappa\}$, $\ms
E_N=\cup_{x\in S}\ms E^x_N$ and let $\Delta_N=E_N \setminus \ms E_N$
so that
\begin{equation}
\label{nv1}
E_N \,=\, \ms E^1_N\cup\dots \cup \ms E^{\kappa}_N
\cup\, \Delta_N \,=\, \ms E_N \, \cup\, \Delta_N \;. 
\end{equation}
Let $\breve{\ms E}^x_N := \ms E_N\setminus \ms E^x_N$,
\begin{equation*}
{\ms E} \;=\; (\ms E_N : N\ge 1) \,,\quad 
{\ms E}^x=(\ms E^x_N : N\ge 1) \quad\textrm{and}\quad 
\breve{\ms E}^x=(\breve{\ms E}^x_N : N\ge 1) \,. 
\end{equation*}

Denote by $\{\eta^{\ms E_N}_t: t\ge 0\}$ the trace of $\{\eta^N_t:
t\ge 0\} $ on $\ms E_N$ \cite{bl2}.  Let $R^{\ms E}_N : \ms E_N \times
\ms E_N\to \bb R_+$ be the transition rates of the trace process
$\{\eta^{\ms E_N}_t : t\ge 0\}$, and denote by $r_N(\ms E^x_N,\ms
E^y_N)$ the mean rate at which the trace process jumps from $\ms
E^x_N$ to $\ms E^y_N$:
\begin{equation}
\label{nv2}
r_N(\ms E^x_N,\ms E^y_N) \; := \; \frac{1}{\mu_N(\ms E^x_N)}
\sum_{\eta\in\ms E^x_N} \mu_N(\eta) 
\sum_{\xi\in\ms E^y_N} R^{\ms  E}_N(\eta,\xi) \;,
\end{equation}
and let
$$
r_N(\ms E^x_N,\breve{\ms E}^x_N) \,=\, \sum_{y\not = x} 
r_N(\ms E^x_N, \ms E^y_N)\,.
$$

For two disjoint subsets $A$, $B$ of $E_N$, denote by $\Cap_N(A,B)$
the capacity between $A$ and $B$, defined in \cite{gl2} by
\eqref{06} for irreducible positive recurrent Markov processes.
For a point ${\bs \xi}_x=(\xi^N_x : N\ge 1)$ in $\ms E^x$, let
\begin{equation*}
\Cap_N(\bs \xi_x) \;=\; \inf_{\eta\in \ms E^x_N} 
\Cap_N (\{\eta\},\{\xi^N_x\})\;.
\end{equation*}

Recall the notion of tunneling behavior and recall properties ({\bf
  M1}), ({\bf M2}), ({\bf M3}) introduced in Definition 2.2 of
\cite{bl2}.

\begin{theorem}
\label{teo1r}
Suppose that there exists a sequence $\bs \theta=(\theta_N : N\ge 1)$
of positive numbers such that, for every pair $x,y\in S$, $x\not= y$,
the following limit exists
\begin{equation}
\tag*{\bf (H0)}
r(x,y) \;:= \; \lim_{N\to\infty} \theta_N\, r_N(\ms E^x_N,\ms E^y_N) \,.
\end{equation}

Suppose that for each $x\in S$, there exists a point $\bs
\xi_x=(\xi^{N}_x : N\ge 1)$ in $\ms E^x$ such that
\begin{equation*}
\tag*{\bf (H1)} 
\lim_{N\to \infty} \frac{\Cap_N(\ms E^x_N,\breve{\ms E}^x_N)}
{\Cap_N(\bs \xi_x)}\;=\;0\; .
\end{equation*}
Then, for any points $\{ {\bs \zeta}_x \in \ms E^x : x\in S\}$,
properties $({\bf M1})$ and $({\bf M2})$ of tunneling hold on the
time-scale $\bs \theta$, with metastates $\{\ms E^x : x\in S\}$,
metapoints $\{\bs \zeta_x : x\in S\}$ and asymptotic Markov dynamics
characterized by the rates $r(x,y)$, $x,y\in S$.

If in addition we have that $\bf(M3)$ holds for each $x\in S$ which is
an absorbing state of the Markov dynamics on $S$ determined by the
rates $r$ and that
\begin{equation*}
\tag*{\bf (H2)}
\lim_{N \to \infty} \frac{ \mu_N(\Delta_N)}{\mu_N(\ms E^{x}_N) } 
\;=\; 0 
\end{equation*}
for each non-absorbing state $x\in S$, then property $\bf (M3)$ holds
for every $x\in S$.
\end{theorem}

In contrast with the reversible case, we are not able to express the
average rates $r_N(\ms E^x_N,\ms E^y_N)$, $x$, $y\in S$, in terms of
capacities. We may, however, formulate a variational problem whose
optimal solution is equal to the average rate. 

\begin{proposition}
\label{s05}
Fix $x\not = y$ in $S$ and consider the variational problem
\begin{equation}
\label{09}
\inf_f \sup_h \big\{ 2\< f, L_N h \>_{\mu_N} - 
\< h, (-L_N) h \>_{\mu_N} \big\},
\end{equation}
where the infimum is carried over all functions $f:E_N\to \bb R$ which
are equal to $1$ on $\ms E^y_N$, $0$ on $\ms E_N \setminus (\ms E^x_N
\cup \ms E^y_N)$, and which are constant on $\ms E^x_N$, while the
supremum is carried over all functions $h:E_N\to \bb R$ which are
equal to $0$ on $\ms E_N \setminus (\ms E^x_N \cup \ms E^y_N)$ and
which are constant over $\ms E^y_N$, $\ms E^x_N$, with possibly
different values at each set. Then, the optimal function $h$
associated to the optimal function $f$ is such that
\begin{equation*}
h(\eta) \;=\; \frac {r_N(\ms E^x_N,\ms E^y_N) } 
{r_N(\ms E^x_N, \breve{\ms E}^x_N)} \quad\text{for all}\quad
\eta\in \ms E^x_N\;.
\end{equation*} 
\end{proposition}

This result is a particular case of Proposition \ref{s02} below.  We
conclude this section with two remarks on metastability. The first one
concerns the metastable behavior of the adjoint of processes which
exhibit a metastable behavior, while the second one points out that
condition ({\bf H0}) alone guarantees the metastable behavior of the
collapsation of the trace process.

\begin{proposition}
\label{s06}
Suppose that a sequence of Markov process $\{\eta^N_t : t\ge 0\}$
satisfies the hypotheses of Theorem \ref{teo1r}, that $\mu_N(\ms
E^x_N)$ converges to a strictly positive limit $m(x)$ for each $x\in
S$ and that Markov dynamics on $S$ determined by the rates $r$ has no
absorbing states.  Then, $m$ is a stationary probability measure for
this asymptotic dynamics, the adjoint process $\{\eta^{*,N}_t : t\ge
0\}$ also satisfies the assumptions of Theorem \ref{teo1r}, and the
limiting jump rates introduced in condition {\rm({\bf H0})} for the
adjoint process, denoted by $r_*(x,y)$, are the adjoint rates of the
asymptotic dynamics with respect to the measure $m$. In particular,
the adjoint process also exhibits a tunneling behavior.
\end{proposition}

\begin{proof}
Denote by $r^*_N(\ms E^x_N, \ms E^y_N)$, $x\not = y\in S$, the mean
jump rates of the trace on $\ms E_N$ of the adjoint process.  By
definition of $m$ and $r$, $\sum_{y\not = x} m(y) r(y,x)$ is equal to
\begin{equation*}
\lim_{N\to\infty} \sum_{y\not = x} \mu_N(\ms E^y_N) \, \theta_N \, r_N(\ms E^y_N,
\ms E^x_N)
\;=\; \lim_{N\to\infty} \sum_{y\not = x} \mu_N(\ms E^x_N) \, \theta_N
\, r^*_N(\ms E^x_N, \ms E^y_N)
\end{equation*}
where the last identity follows from the explicit expressions for the
jump rates of the trace processes presented in \cite[Proposition
6.1]{bl2}. By \eqref{a03} and \eqref{02}, 
\begin{equation*}
\mu_N(\ms E^x_N) \sum_{y\not = x} r^*_N(\ms E^x_N, \ms
E^y_N) \;=\; \mu_N(\ms E^x_N) r^*_N(\ms E^x_N, \breve{\ms E}^x_N) \;=\;
\Cap^*_N(\ms E^x_N, \breve{\ms E}^x_N)\;,
\end{equation*}
where $\Cap^*_N$ represents the capacity with respect to the adjoint
process. By \cite[Lemma 2.3]{gl2} and by formula (2.4) in that
article, $\Cap^*_N(\ms E^x_N, \breve{\ms E}^x_N) = \Cap_N(\ms E^x_N,
\breve{\ms E}^x_N)$ so that
\begin{equation*}
\lim_{N\to\infty} \sum_{y\not = x} \mu_N(\ms E^y_N) \, \theta_N \, r_N(\ms E^y_N,
\ms E^x_N) \;=\; \lim_{N\to\infty} \mu_N(\ms E^x_N) \sum_{y\not = x}  
\theta_N \, r_N(\ms E^x_N, \ms E^y_N)\;,
\end{equation*}
that is, $\sum_{y\not = x} m(y) r(y,x) = m(x) \sum_{y\not = x}
r(x,y)$, proving that $m$ is a stationary measure.

To prove the second assertion of the proposition, note that condition
({\bf H1}) is satisfied because the measure and the capacities of the
adjoint process coincides with the measure and the capacities of the
original process. 

On the other hand, as we have seen above,
\begin{equation*}
\mu_N(\ms E^x_N) \, r^*_N(\ms E^x_N, \ms E^y_N) \;=\;
\mu_N(\ms E^y_N) \, r_N(\ms E^y_N, \ms E^x_N) 
\end{equation*}
for all $x\not = y$ in $S$. Hence, by condition ({\bf H0}) for the
original process and by assumption,
\begin{equation*}
\lim_{N\to\infty} \theta_N \, r^*_N(\ms E^x_N, \ms E^y_N)
\;=\; \frac {m(y)}{m(x)} r(y,x) \;=:\; r_*(x,y) \;.
\end{equation*}
This proves condition ({\bf H0}) for the adjoint process and 
shows that the dynamics on $S$ induced by the rates $r_*(x,y)$ has no
absorbing point since
\begin{equation*}
\sum_{y\not = x} r_*(x,y) \;=\; \sum_{y\not = x} r(x,y) \;>\; 0
\end{equation*}
by assumption. In particular, condition ({\bf H2}) for the adjoint
process is in force.
\end{proof}

Recall that we denote by $\{\eta^{\ms E_N}_t : t\ge 0\}$ the trace of
the process $\eta^N_t$ on $\ms E_N$. Denote furthermore by
$\{\eta^{S}_t : t\ge 0\}$ the Markov process on $S=\{1, \dots,
\kappa\}$ obtained from $\eta^{\ms E_N}_t$ by collapsing each set $\ms
E^x_N$ to the point $x$ \cite{gl2}. The jump rates of the Markov
process $\eta^{S}_t$, denoted by $r^S_N(x,y)$, are given by
\begin{equation*}
r^S_N(x,y) \;=\; r_N(\ms E^x_N , \ms E^y_N)\;, \quad x\not = y \in S\;. 
\end{equation*}
In particular, under assumption ({\bf H0}) the speeded up Markov
process $\eta^{S}_{t\theta_N}$ converges to the Markov process on $S$
characterized by the rates $r(x,y)$.

\section{Proof of Theorem \ref{teo1r}}
\label{sec1}

We prove in this section the main result of the article.  Given two
disjoint subsets $A,B\subset E_N$, let
\begin{equation*}
f_{AB}(\eta) \;:=\; 
\mb P_{\eta}\big[\, H_{A} < H_B \,\big]\;, \quad
f_{AB}^*(\eta) \;:=\; 
\mb P^*_{\eta}\big[\, H_{A} < H_B \,\big]\;.
\end{equation*}
In addition, for point $\bs \eta=(\eta^N : N\ge 1)$, $\bs \xi=(\xi^N :
N\ge 1)$ in $\ms E$, $\eta^N\not=\xi^N$, and a set $\ms A = (A_N:
N\ge 1)$, set 
\begin{equation*}
f^*_N({\bs \eta},{\bs \xi}) \,=\, f^*_{\{\eta^N\}\{\xi^N\}}
\;,\;\;
f^*_N({\bs \eta},\ms A) \,=\, f^*_{\{\eta^N\} A_N}
\;,\;\;
\Cap_N({\bs \eta},{\bs \xi})=\Cap_N(\{\eta^N\},\{\xi^N\})\;,
\end{equation*}
and a similar notation without the upper index $*$.

For two subsets $A$, $B$ of $E_N$, denote by $H_A(B)$ the time the
process $\eta^N_t$ spent on the set $B$ before hitting the set $A$:
\begin{equation*}
H_A(B) \;=\; \int_0^{H_A} \mb 1\{\eta^N_s \in B\}\, ds\;.
\end{equation*}

Fix $x\in S$. By the definition of $H_{\xi^N_x}(\ms E^x_N)$ given just
above, by \eqref{g02} below, and by identity \eqref{a03} for the
capacity,
\begin{equation*}
\begin{split}
& {\mb E}_{\eta^N} \Big[\int_{0}^{H_{\mb \xi_x}} 
R_N(\eta^N_s, \breve{\ms E}^x_N)\,{\bf 1}\{\eta^N_s\in  \ms E^x_N\} 
\,ds \Big] 
\;=\; \frac{\langle \, R_N(\,\cdot\,, \breve{\ms E}^x_N) \,  
{\bf 1}\{\ms E^x_N\} \,,\, f^*_N({\bs \eta},{\bs \xi}) \, 
\rangle_{\mu_N}}{\Cap_N( {\bs \eta} , {\bs \xi })}\;, \\
&\quad
r_N(\ms E^x_N,\breve{\ms E}^x_N) \, {\mb E}_{\eta^N}[H_{\xi^N_x}(\ms E^x_N)] 
\;=\; \frac{ \Cap_N(\ms E^x_N,\breve{\ms E}^x_N) }
{ \mu_N(\ms E^x_N) \, }\,
\frac{  \langle \, {\bf 1}\{\ms E^x_N\} \,,\, 
f^*_N({\bs \eta},{\bs \xi}) \,\rangle_{\mu_N} }
{ \Cap_N({\bs \eta},{\bs \xi}) }  \;\cdot
\end{split}
\end{equation*} 
Since $0 \le f^*_N({\bs \eta},{\bs \xi}) \le 1$, and since $\langle \,
R_N(\,\cdot\,, \breve{\ms E}^x_N) \, {\bf 1}\{\ms E^x_N\} \,
\rangle_{\mu_N} = \mu_N(\ms E^x_N) r_N(\ms E^x_N,\breve{\ms E}^x_N)$,
by identity \eqref{a03} again, we obtain the following estimates which
will be used repeatedly below
\begin{equation}
\label{g03}
\begin{split}
& {\mb E}_{\eta^N} \Big[\int_{0}^{H_{\mb \xi_x}} 
R_N(\eta^N_s, \breve{\ms E}^x_N)\,{\bf 1}\{\eta^N_s\in  \ms E^x_N\} 
\,ds \Big] 
\,\le \, \frac{\Cap_N(\ms E^x_N,\breve{\ms E}^x_N)} 
{\Cap_N({\bs \xi}_x)}\;,  \\
&\quad
r_N(\ms E^x_N,\breve{\ms E}^x_N) \, {\mb E}_{\eta^N}[H_{\xi^N_x}(\ms
E^x_N)] 
\,\le \,  \frac{\Cap_N(\ms E^x_N,\breve{\ms E}^x_N)} {\Cap_N({\bs
    \xi}_x)}\; \cdot
\end{split}
\end{equation}

Lemma \ref{s04} below is Proposition 5.10 of \cite{bl2}. The proof
needs some modifications to take into account the lack of
reversibility.  Let $R^{*,\ms E_N}_N : \ms E_N \times \ms E_N\to \bb
R_+$ be the transition rates of the trace on $\ms E_N$ of the adjoint
process $\eta^{*,N}_t$.  Denote by $r^*_N(\ms E^x_N,\ms E^y_N)$ the
mean rate at which this trace process jumps from $\ms E^x_N$ to $\ms
E^y_N$:
\begin{equation*}
r^*_N(\ms E^x_N,\ms E^y_N) \; := \; \frac{1}{\mu_N(\ms E^x_N)}
\sum_{\eta\in\ms E^x_N} \mu_N(\eta) 
\sum_{\xi\in\ms E^y_N} R^{*,\ms  E_N}_N(\eta,\xi) 
\;,
\end{equation*}
and let $R^{*,\ms  E_N}_N(\eta, \ms E^y_N) = \sum_{\xi\in\ms E^y_N}
R^{*,\ms  E_N}_N(\eta,\xi)$, $y\not = x$,
$$
r^*_N(\ms E^x_N,\breve{\ms E}^x_N) \,=\, \sum_{y\not = x} 
r^*_N(\ms E^x_N,\ms E^y_N)\,.
$$

\begin{lemma}
\label{s04} 
Fix $x\in S$.  Assume condition {\rm ({\bf H1})} for some point
$\bs \xi_x=(\xi^{N}_x : N\ge 1)$ in $\ms E^x$. Then, for every point
$\bs\zeta=(\zeta^N:N\ge 1)$ in $\ms E^x$,
\begin{equation}
\label{ce}
\lim_{N\to \infty} \frac{ \Cap_N(\ms E^x_N, \breve{\ms E}^x_N) }
{\Cap_N(\zeta^N ,  \breve{\ms E}^x_N ) } \;=\; 1\;,
\end{equation}
and
\begin{equation}
\label{vze}
\lim_{N\to \infty}\inf_{\eta\in \ms E^x} 
{\bf P}^*_{\eta}[\, H_{\zeta^N} < H_{\breve{\ms E}^x_N} \,] \;=\; 1
\;, \quad
\lim_{N\to \infty}\inf_{\eta\in \ms E^x} 
{\bf P}_{\eta}[\, H_{\zeta^N} < H_{\breve{\ms E}^x_N} \,] \;=\; 1 \;.
\end{equation}
\end{lemma} 

\begin{proof}
Fix $x\in S$ and some point $\bs \xi_x=(\xi^{N}_x : N\ge 1)$ in $\ms
E^x$ for which ({\bf H1}) holds.  In view of \eqref{a02}, by
\eqref{g03} for the adjoint process and by assumption ({\bf H1}), for
every point $\bs \eta =(\eta^{N} : N\ge 1)$ in $\ms E^x$,
\begin{equation*}
\begin{split}
&\lim_{N\to\infty} {\mb E}^*_{\eta^N} \Big[ \int_{0}^{H_{\xi^N_x}} 
R^{*,\ms E_N}_N(\eta^{*,N}_s, \breve{\ms E}^x_N)\,{\mb 1}
\{\eta^{*,N}_s\in  \ms E^x_N\} \,ds \Big] \,= \, 0\;,  \\
&\quad \lim_{N\to\infty}
r^*_N(\ms E^x_N,\breve{\ms E}^x_N) \, {\mb E}^*_{\eta^N}
[H_{\xi^N_x}(\ms E^x_N)] \,=\, 0\;.
\end{split}
\end{equation*}
The assumptions of \cite[Proposition 5.1]{bl2} are thus fulfilled for
the adjoint process and with $\ms W=\ms E^x$, $\ms B^c = \breve{\ms
  E}^x$ and the point $\bs \xi_x=(\xi^{N}_x : N\ge 1)$. 

Fix an arbitrary point ${\bs \zeta} = (\zeta^N : N\ge 1)$ in $\ms
E^x$. We show that conditions (i) and (ii) of \cite[Proposition
4.2]{bl2} are in force for the adjoint process and with $\ms W=\ms
E^x$, $\ms B^c = \breve{\ms E}^x$, $\bs \xi = \bs \zeta$.

Fix a point ${\bs \eta}=(\eta^N : N\ge 1)$ in $\ms E^x$. By
\eqref{g02} applied to $\eta^N$, $g= \mb 1\{\ms E^x_N\}$,
$B=\{\xi^N\}$, and to $\xi^N$, $g= \mb 1\{\ms E^x_N\}$,
$B=\{\zeta^N\}$,
\begin{eqnarray*}
{\bf E}^*_{\eta^N} [\, H_{\zeta^N}(\ms E^x_N) \,] &\le & 
{\bf E}^*_{\eta^N} [\, H_{\xi^N}(\ms E^x_N) \,] \,+\, 
{\bf E}^*_{\xi^N} [\, H_{\zeta^N}(\ms E^x_N) \,] \\
&\le&  \frac{\mu_N(\ms E^x_N)}{\Cap_N(\eta_N , \xi_N)} 
{\bs 1}\{ \eta^N\not = \xi^N\} \,+\, 
\frac {\mu_N(\ms E^x_N)}{\Cap_N(\zeta_N, \xi_N)} 
{\bs 1}\{ \zeta^N\not = \xi^N\} \\
&\le& \frac{ 2 \,\mu_N(\ms E^x_N) }{ \Cap_N(\xi_N) }\; \cdot
\end{eqnarray*}
It follows from this estimate, identities (\ref{a03}), \eqref{a02} and
assumption ({\bf H1}), that
\begin{equation*}
\lim_{N\to\infty} r^*_N(\ms E^x_N,\breve{\ms E}^x_N) \, {\bf E}^*_{\eta^N} 
[\, H_{\zeta^N}(\ms E^x_N) \,] \;=\; 0\,.
\end{equation*}
Therefore, applying the assertion (5.1) of \cite[Proposition
5.1]{bl2}, to $\ms W=\ms E^x$, $\ms B^c = \breve{\ms E}^x$,
\begin{equation*}
\lim_{N\to\infty} \frac {{\bf E}^*_{\eta^N} [\, H_{\zeta^N}(\ms E^x_N)
  \,]} {{\bf E}^*_{\eta^N}[H_{\breve{\ms E}^x_N}(\ms E^x_N)]} \;=\;
\lim_{N\to\infty} r^*_N(\ms E^x_N,\breve{\ms E}^x_N)\,
{\bf E}^*_{\eta^N} [\, H_{\zeta_N}(\ms E^x_N)  \,] \;=\; 0 \;.
\end{equation*}
This result shows that condition (i) of \cite[Proposition 4.2]{bl2}
holds.

Fix a point ${\bs \eta}=(\eta^N : N\ge 1)$ in $\ms E^x$. By
\cite[Proposition 5.1]{bl2}, under $\mb P^*_{\eta^N}$, $H_{\breve{\ms
    E}^x_N}(\ms E^x_N)/ {\bf E}^*_{\eta^N}[H_{\breve{\ms E}^x_N}(\ms
E^x_N)]$ converges to a mean one exponential distribution, which is
condition (ii) of \cite[Proposition 4.2]{bl2}.

By items (i) and (ii) of Proposition 4.2 we conclude that $(\ms
E^x,\ms E^x,{\bs \zeta})$ is a valley for the trace on $\ms E_N$ of
the adjoint process. Hence, for any point ${\bs \eta}=(\eta^N : N\ge
1)$ in $\ms E^x$
$$
\lim_{N\to\infty}{\bf P}^*_{\eta^N}[\,H_{\zeta^N}(\ms E_N) 
< H_{\breve{\ms E}^x_N}(\ms E_N) \,]\,=\,1\,,
$$
which implies condition {\bf (V1)} for the triple $(\ms E^x, \ms E,
{\bs \zeta})$ because clearly $\{H_{\zeta^N}(\ms E_N) < H_{\breve{\ms
    E}^x_N}(\ms E_N)\} \subseteq \{H_{\zeta^N} < H_{\breve{\ms
    E}^x_N}\}\,$ ${\bf P}^*_{\eta^N}$-a.s., proving the first
assertion in (\ref{vze}). The proof of the second one is exactly the
same replacing the adjoint process by the original one.

The arguments presented in the first part of the proof, applied to the
original process instead of the adjoint process, show that the
assumptions of \cite[Proposition 5.1]{bl2} are in force for the
original process and with $\ms W=\ms E^x$, $\ms B^c = \breve{\ms E}^x$
and the point $\bs \xi_x=(\xi^{N}_x : N\ge 1)$. Fix a point ${\bs
  \zeta}=(\zeta^N : N\ge 1)$ in $\ms E^x$. By \eqref{g02} with
$A=\{\zeta^N\}$, $B = \breve{\ms E}^x_N$ and $g=\mb 1\{\ms E^x_N\}$,
and by identity (\ref{a03}), the limit (5.1) in \cite[Proposition
5.1]{bl2} can be re-written as
$$
\lim_{N\to \infty} \frac{\langle {\bs 1}\{\ms E^x_N\} , 
f^*_N(\bs \zeta, \breve{\ms E}^x) \rangle_{\mu_N}\, 
\Cap_N(\ms E^x_N, \breve{\ms E}^x_N)}{\mu_N(\ms E^x_N) \, 
\Cap_N(\zeta^N, \breve{\ms E}^x_N)} 
\;=\; 1\,.
$$
By (\ref{vze}), the infimum of $f^*_N( {\bs \zeta}, \breve{\ms E}^x)$
over $\ms E^x_N$ converges to $1$ as $N\uparrow \infty$. Therefore,
(\ref{ce}) follows from this observation and the previous identity.
\end{proof}

\begin{proof}[Proof of Theorem \ref{teo1r}]
The proof of this result is similar to the one of \cite[Theorem
2.7]{bl2}.  We first need to check that all assumptions of
\cite[Theorem 2.4]{bl2} are in force.  

By \eqref{g03}, conditions ({\bf C1}) and ({\bf C2}) in \cite[Theorem
2.4]{bl2} follow from the stronger condition ({\bf H1}). Therefore, by
\cite[Theorem 2.4]{bl2}, properties $({\bf M1})$ and $({\bf M2})$ of
tunneling hold on the time-scale $\bs \theta$, with metastates $\{\ms
E^x : x\in S\}$, metapoints $\{\bs \xi_x : x\in S\}$ and asymptotic
Markov dynamics characterized by the rates $r(x,y)$, $x,y\in S$.  The
estimate \eqref{vze} shows that we can replace the metapoint $\bs
\xi_x$ by any point $\bs \zeta_x$ in $\ms E^x$.

We turn to condition ({\bf M3}). Fix $x\in S$ and $t >0$. Decomposing
the time interval $[0, t\theta_N]$ according to the successive visits
of the trace process $\eta^{\ms E_N}_t$ to the metastates $\ms E^z_N$,
$z\in S$, we see that to prove condition ({\bf M3}) for all $x\in S$,
it is enough to show that this condition is in force for each
absorbing state $x$ for the Markov dynamics on $S$ determined by the
rates $r$ and that
\begin{equation*}
\limsup_{N \to \infty} \sup_{\eta\in \ms E^{x}_N} \frac{1}{\theta_N} 
{\bf E}_{\eta} \big[\,  H_{\breve{\ms E}^{x}_N}(\Delta_N) \,\big] 
\,=\, 0
\end{equation*}
holds for each non-absorbing state $x$. 

Fix a non-absorbing state $x\in S$. By \eqref{g02} and since
$f^*_{\eta \breve{\ms E}^{x}_N}$ is bounded by $1$, the expectation is
less than or equal to $\mu_N(\Delta_N)/\Cap_N(\eta, \breve{\ms
  E}^{x}_N)$. By Lemma \ref{s04}, we may replace asymptotically $\eta$
by $\ms E^x_N$ in the previous ratio. Since $\Cap (\ms E^x_N,
\breve{\ms E}^{x}_N)=\mu_N(\ms E^x_N) r_N(\ms E^x_N,\breve{\ms
  E}^{x}_N)$, we have shown that
\begin{equation*}
\limsup_{N \to \infty} \sup_{\eta\in \ms E^{x}_N}
\frac{1}{\theta_N} {\bf E}_{\eta} \big[\,  
H_{\breve{\ms E}^{x}_N}({\Delta_N}) \,\big] \;\le\;
\limsup_{N \to \infty} \frac 1{\theta_N \, r_N(\ms E^x_N, \breve{\ms E}^x_N)}
\, \frac{ \mu_N(\Delta_N) }{ \mu_N(\ms E^{x}_N) }\; \cdot
\end{equation*}
As $x$ is a non-absorbing point, by assumptions ({\bf H0}), ({\bf
  H2}), the right hand side vanishes, which concludes the proof of the
theorem.
\end{proof}

\section{A formula for the average jump rate}

Consider an irreducible chain $\{\eta(t) : t\ge 0\}$ on a countable
state space $E$. Let $B\subset F$ be two non-empty subsets of
$E$. Denote by $\{\eta^C (t) : t\ge 0\}$ the Markov process on
$(E\setminus B) \cup \{\mf d\}$, where $\mf d \not\in E$ is an extra
pointed added to represent the collapsed set $B$, obtained from $\eta
(t)$ by collapsing the set $B$ to a point $\mf d$ \cite{gl2}, and
denote by $\{\eta^T (t) : t\ge 0\}$ the trace of the process $\eta(t)$
on the set $F$ \cite{bl2}. Denote furthermore by $\{\eta^{TC} (t) :
t\ge 0\}$ the trace of the Markov process $\eta^C(t)$ on the set $G =
(F\setminus B) \cup \{\mf d\}$, and by $\{\eta^{CT} (t) : t\ge 0\}$
the process $\eta^T(t)$ where the set $B$ has been collapsed to a
point $\mf d$. Note that the state space of both processes is $G =
(F\setminus B) \cup \{\mf d\}$

\begin{lemma}
\label{s03}
The processes $\{\eta^{TC} (t) : t\ge 0\}$ and $\{\eta^{CT} (t) : t\ge
0\}$ have the same law.
\end{lemma}

\begin{proof}
Both processes are Markov processes on $G= (F\setminus B) \cup \{\mf
d\}$. To prove the assertion it is therefore enough to check that the
jump rates are equal. 

Denote by $R_C$, $R_T$, $R_{CT}$, $R_{TC}$ the jump rates of the
Markov processes $\eta^C(t)$, $\eta^T(t)$, $\eta^{CT}(t)$,
$\eta^{TC}(t)$, respectively. Let $\eta\in F\setminus B$. We show that
$R_{TC}(\mf d, \eta) = R_{CT}(\mf d, \eta)$ and leave the other
identities to the reader.

By \cite[Proposition 6.1]{bl2}, $R_{TC}(\mf d, \eta) = \lambda_C(\mf
d) \mb P_{\mf d}^C[ H^+_{G} = H^+_\eta]$, where $\mb P_{\mf d}^C$
stands for the probability measure on the path space $D(\bb R_+, 
(E\setminus B) \cup \{\mf d\})$ induced by the Markov process
$\eta^C(t)$ starting from $\mf d$. The probability $\mb P_{\mf d}^C[
H^+_{G} = H^+_\eta]$ can be written as the sum of the probabilities of
all paths connecting $\mf d$ to $\eta$ through points which do not
belong to $G$:
\begin{equation*}
\mb P_{\mf d}^C \big[ H^+_{G} = H^+_\eta \big] \;=\;
\sum_{\gamma} p_C(\mf d, \xi_1) \cdots p_C(\xi_{n-1}, \eta)\;,
\end{equation*}
where the sum is carried over all paths $\gamma = (\xi_0=\mf d, \xi_1,
\dots, \xi_{n-1}, \xi_n = \eta)$ such that $\xi_j\not\in G$, $1\le
j\le n-1$ or, equivalently, such that $\xi_j\in E\setminus F$ . In
this formula, $p_C$ stands for the jump probabilities of the collapsed
chain. Hence,
\begin{equation*}
R_{TC}(\mf d, \eta) \;=\; 
\sum_{\gamma} R_C(\mf d, \xi_1) p_C(\xi_1, \xi_2) \cdots
p_C(\xi_{n-1}, \eta)\;. 
\end{equation*}
By equation (3.3) in \cite{gl2}, $p_C(\xi, \zeta) = p(\xi,\zeta)$ if
$\xi$, $\zeta \in E\setminus B$ and $R_C(\mf d, \xi_1) = \mu(B)^{-1}
\sum_{\zeta\in B} \mu(\zeta) R(\zeta,\xi_1)$. The previous expression
is thus equal to
\begin{equation*}
\begin{split}
R_{TC}(\mf d, \eta) \; & =\; \frac 1{\mu(B)} \sum_{\zeta\in B} \mu(\zeta)
\sum_{\gamma} R(\zeta, \xi_1) p(\xi_1, \xi_2) \cdots
p(\xi_{n-1}, \eta) \\
\; & =\; \frac 1{\mu(B)} \sum_{\zeta\in B} \mu(\zeta) \lambda(\zeta)
\mb P_\zeta \big[ H^+_{F} = H^+_\eta \big] \;. 
\end{split}
\end{equation*}

On the other hand, by equation (3.3) in \cite{gl2} and Proposition 6.1
in \cite{bl2}, for $\eta\in F\setminus B$,
\begin{equation*}
R_{CT}(\mf d, \eta) \; =\; \frac 1{\mu(B)} \sum_{\zeta\in B}
\mu(\zeta) R_T(\zeta, \eta) \; =\;
\frac 1{\mu(B)} \sum_{\zeta\in B} \mu(\zeta) \lambda(\zeta)
\mb P_\zeta \big[ H^+_{F} = H^+_\eta \big]\;.
\end{equation*}
This proves the lemma.
\end{proof}

In the next result we consider an irreducible Markov process $\{\eta
(t) : t\ge 0\}$ on a finite state space $E$. Let $A_0$, $A_1$, $B$ be
three disjoint subsets of $E$ and let $R_T$ denote the jump rates of
$\eta^T(t)$, the trace of the process $\eta(t)$ on the set $F=A_0\cup
A_1\cup B$. We denote by $r_T(B, A_i)$, $i=0$, $1$, the average rate
at which the trace process jumps from $B$ to $A_i$:
\begin{equation*}
r_T(B, A_i) \;=\; \frac 1{\mu(B)} \sum_{\eta\in B} \mu(\eta)
\sum_{\xi\in A_i} R_T (\eta,\xi)
\;=\; \frac 1{\mu(B)} \sum_{\eta\in B} \mu(\eta) R_T (\eta,A_i)\;.
\end{equation*}

\begin{proposition}
\label{s02}
Consider the variational problem
\begin{equation}
\label{03}
\inf_f \sup_h \big\{ 2\< f, Lh \>_\mu - \< h, (-L) h \>_\mu \big\},
\end{equation}
where the infimum is carried over all function $f:E\to \bb R$ which
are constant at each set $A_0$, $A_1$, $B$, equal to $1$ at $A_1$ and
equal to $0$ at $A_0$, while the supremum is carried over all
functions $h$ which are equal to $0$ at $A_0$ and which are constant
over $A_1$, $B$, with possibly different values at each set. Then, the
optimal function $h$ associated to the optimal function $f$ is such
that
\begin{equation*}
h(B) \;=\; \frac {r_T(B,A_1) } {r_T(B,A_0) + r_T(B,A_1)}\;\cdot
\end{equation*}
\end{proposition}

\begin{proof}
Denote by $L_C$ the generator of the Markov process $\eta^C(t)$ on
$E_B=(E\setminus B) \cup \{\mf d\}$, where the set $B$ has been collapsed
to a point $\mf d \not \in E$. For a function $g:E\to \bb R$ constant
over $B$, denote by $\overline{g} : E_B \to \bb R$ the function which
is equal to $g$ at $E\setminus B$, and such that $\overline{g} (\mf d)
= g(B)$. 

Denote by $\overline{\mu}$ the probability measure on $E_B$ equal to
$\mu$ on $E\setminus B$ and such that $\overline{\mu} (\mf d) =
\mu(B)$. A calculation, performed below equation (3.7) in \cite{gl2},
shows that for any pair of functions $f$, $h: E \to\bb R$ constant
over $B$,
\begin{equation}
\label{04}
\< f, Lh \>_\mu \;=\; \< \overline{f}, L_C \overline{h}
\>_{\overline{\mu}} \;.
\end{equation}
Conversely, given $\overline{f}$, $\overline{h}:E_B\to\bb R$, if we
define $f$, $h: E\to\bb R$ to be equal to $f$, $h$ on $E\setminus B$
and such that $f(\eta) = \overline{f} (\mf d)$, $h (\eta) =
\overline{h} (\mf d)$, $\eta\in B$, \eqref{04} holds. 

It follows from \eqref{04} that the variational problem \eqref{03} is
equivalent to the variational problem
\begin{equation*}
\inf_{\overline{f}} \sup_{\overline{h}} 
\big\{ 2\< \, \overline{f} , L_C \overline{h} \, \>_{\overline{\mu}}
\;-\; \< \, \overline{h} , (-L_C) \, \overline{h} \,
\>_{\overline{\mu}} \big\},
\end{equation*}
where the infimum is carried over all functions $\overline{f}:E_B\to
\bb R$ which are equal to $1$ at $A_1$ and equal to $0$ at $A_0$,
while the supremum is carried over all functions $h$ which are
constant over $A_0$, $A_1$. This variational problem corresponds to
the variational problem for the capacity between $A_1$ and $A_0$ for
the Markov process $\eta^C(t)$. By \cite[Theorem 2.4]{gl2}, the
optimal function $\overline{h}: E_B\to \bb R$ associated to the
optimal function $\overline{f}$ is the harmonic function $\overline{h}
(\eta) = \mb P^C_\eta[H_{A_1} < H_{A_0}]$. In particular,
$\overline{h} (\mf d) = \mb P^C_{\mf d} [H_{A_1} < H_{A_0}]$ and the
optimal function $h: E\to \bb R$ associated to the optimal function
$f$ for the variational problem \eqref{03} is such that
\begin{equation}
\label{05}
h(B) \;=\; \mb P^C_{\mf d} \big[H_{A_1} < H_{A_0} \big]\;.
\end{equation}

Decompose the event $\{H_{A_1} < H_{A_0}\}$ according to whether
$H^+_{\mf d} < H_{A_1}$ or $H^+_{\mf d} > H_{A_1}$ and use the strong
Markov property to get that
\begin{equation*}
\mb P^C_{\mf d} \big[ H_{A_1} < H_{A_0} \big]
\;=\; \frac {\mb P^C_{\mf d} \big[ H_{A_1} < 
\min\{ H^+_{\mf d} , H_{A_0}\} \big]}{1 - \mb P^C_{\mf d} \big[ H^+_{\mf d} < 
\min\{H_{A_1}  , H_{A_0}\} \big]}\;\cdot
\end{equation*}
The denominator can be written as $\mb P^C_{\mf d} [H_{A_1} < \min\{
H^+_{\mf d} , H_{A_0}\}] + \mb P^C_{\mf d} [H_{A_0} < \min\{ H^+_{\mf
  d} , H_{A_1}\}]$. Taking the trace of the process $\eta^C(t)$ on the
set $A_0 \cup A_1 \cup \{\mf d\}$, we get that
\begin{equation*}
\mb P^C_{\mf d} \big[ H_{A_1} < H_{A_0} \big]
\;=\; \frac {\mb P^{TC}_{\mf d} \big[ H_{A_1} < 
\min\{ H^+_{\mf d} , H_{A_0}\} \big]}
{\mb P^{TC}_{\mf d} \big[ H_{A_0} < 
\min\{ H^+_{\mf d} , H_{A_1}\} \big] +
\mb P^{TC}_{\mf d} \big[ H_{A_1} < 
\min\{ H^+_{\mf d} , H_{A_0}\} \big]}\;\cdot
\end{equation*}
By Lemma \ref{s03}, we may replace the probability $\mb P^{TC}_{\mf
  d}$ by the probability $\mb P^{CT}_{\mf d}$ in this
formula. Since $\mb P^{TC}_{\mf d} [ H_{A_0} < 
\min\{ H^+_{\mf d} , H_{A_1}\} ]$ is equal to $p_{CT}(\mf d,A_0)$,
where $p_{CT}$ stands for the jump probabilities of the process
$\eta^{CT}(t)$, multiplying the denominator and the numerator by
$\lambda_{CT}(\mf d)$, the previous ratio becomes
\begin{equation*}
\frac {R_{CT} (\mf d,A_1)}
{R_{CT} (\mf d,A_0) + R_{CT} (\mf d,A_1)}\;\cdot
\end{equation*}
By \cite[Proposition 6.1]{bl2} and since the stationary measure for
the trace process is the conditional measure, this expression is equal
to
\begin{equation*}
\frac { \mu(B)^{-1} \sum_{\eta\in B} \mu(\eta) R_T (\eta,A_1)}
{\mu(B)^{-1} \sum_{\eta\in B} \mu(\eta) R_T (\eta,A_0) 
+ \mu(B)^{-1} \sum_{\eta\in B} \mu(\eta) R_T (\eta,A_1)}\;\cdot
\end{equation*}
By definition of $r_T(B,A_i)$ we conclude that
\begin{equation*}
\mb P^C_{\mf d} \big[ H_{A_1} < H_{A_0} \big]
\;=\; \frac {r_T(B,A_1) } {r_T(B,A_0) + r_T(B,A_1)}\;\cdot
\end{equation*}
This identity concludes the proof of the proposition in view of
\eqref{05}.
\end{proof} 

\section{Birth and death process}

We prove in this subsection the metastable behavior of a class of
birth and death processes which includes the zero range processes
examined in \cite{bl3} evolving on two sites. Proposition \ref{bd}
presented below was used in \cite{bl3} to prove the metastability of
reversible zero range processes. 

Fix $a<b$ in $\bb R$ and consider a nonnegative smooth function
$H:[a,b] \to \bb R_+$. Assume that $H$ vanishes only at a finite
number of points denoted by $a_1<a_2< \cdots < a_m$:
\begin{equation*}
H(x) = 0 \quad\text{if and only if}\quad x\in\{a_1, \dots, a_m\}\;.
\end{equation*}
We do not exclude the possibility that $H$ vanishes at the boundary
points $a$, $b$. 

For each $i=1,\dots,m$, assume that there exist a neighborhood
$V_{a_i}$ of $a_i$ and $\alpha_i >0$ such that
\begin{equation*}
  H (x) = |x-a_i|^{\alpha_i} \quad \text{for all $x\in V_{a_i}$}\;,
\end{equation*}
and that $V_{a_i}\cap V_{a_j}=\varnothing$ for $i\not = j$. Let $\alpha =
\max\{\alpha_i : 1\le i\le m\}$. Assume that $\alpha >1$ and that
there are at least two exponents $\alpha_i$ equal to $\alpha$:
\begin{equation*}
\kappa \;:=\; \big| \{i : \alpha_i = \alpha\}\big| \;\ge\; 2\;,
\end{equation*}
where $|A|$ indicates the cardinality of a finite set $A$. Denote by
$b_1 < b_2 < \cdots <b_\kappa $ the elements of $\{a_1, \dots, a_m\}$
whose associated exponents are $\alpha$.

The definition of the state space of the birth and death process
requires some notation. Fix $N\ge 1$, $1\le i\le m-1$, and let 
\begin{eqnarray*}
\!\!\!\!\!\!\!\!\!\!\!\!\! &&
k^N_{0} \;=\; \min\Big\{ k\ge 0 \;:\; a > a_{1} 
- (k+1)/N \Big\}\; , \\
\!\!\!\!\!\!\!\!\!\!\!\!\! && \quad
k^N_{i} \;=\; \min\Big\{ k\ge 0 \;:\; a_i + (k+1)/N > a_{i+1} 
- (k+1)/N \Big\}\;, \\
\!\!\!\!\!\!\!\!\!\!\!\!\! && \qquad
k^N_{m} \;=\; \min\Big\{ k\ge 0 \;:\; a_m + (k+1)/N > b \Big\}\;.
\end{eqnarray*}
For $1\le i\le m-1$, set 
\begin{eqnarray*}
\!\!\!\!\!\!\!\!\!\!\!\!\! &&
G_{N,0} \;=\; \big\{a_1 - k^N_0/N, \dots, a_1 - 1/N, a_1\big\}\;, \\
\!\!\!\!\!\!\!\!\!\!\!\!\! && \quad 
G_{N,i} \;=\; \big\{a_i +1/N, a_i + 2/N, \dots, a_i + k^N_i/N, 
a_{i+1} - k^N_i/N, \dots, a_{i+1}\big\}\;, \\
\!\!\!\!\!\!\!\!\!\!\!\!\! && \qquad 
G_{N,m} \;=\; \big\{a_m+1/N, a_m+2/N, \dots, a_m +k^N_m/N\big\}
\end{eqnarray*}
and let the state space $E_N$ be $\bigcup_{i=0}^m G_{N,i}$. Note that
the exact definition of $E_N$ is not important for the meta-stability
behavior discussed in this section. The elements of $E_N$ are denoted
by the letters $x$, $y$, $z$. Two points $x<y$ are said to be
neighbors in $E_N$ if there is no $z$ in $E_N$ such that $x<z<y$.

Let $\nu_N$ be the probability measure on $E_N$ defined by
\begin{equation*}
\nu_N(x) \;=\; 
\left\{
\begin{array}{ll}
\displaystyle{
\frac 1{Z_N} \frac 1{H(x)} }& \text{if } x\not\in \{a_1, \dots,
a_m\}\;, \\
\displaystyle{
\frac 1{Z_N} N^{\alpha_i}} & \text{if } x= a_i \text{ for some $1\le
  i\le m$} \;.
\end{array}
\right.
\end{equation*}
In this formula $Z_N$ is a normalizing constant. An elementary
computation shows that
\begin{equation}
\label{ef04}
\lim_{N\to\infty} \frac{Z_N}{N^\alpha} \;=\;
\sum_{i=1}^\kappa \Big\{ 1 + \sigma_i \sum_{k\ge 1} \frac
1{k^\alpha}\Big\} \;.
\end{equation}
where $\sigma_i=1$ if $b_i\in \{a,b\}$ and $\sigma_i=2$ otherwise. In
particular, if we denote
\begin{equation*}
 m(b_i) \;=\; 1 + \sigma_i \sum_{k\ge 1} \frac 1{k^\alpha}\;, 
\quad \text{for } i=1, \dots, \kappa
\end{equation*}
then
\begin{equation*}
\lim_{N\to\infty} \nu_N(b_i) \;=\; 
\Big\{ \sum_{j=1}^{\kappa} m(b_j) \Big\}^{-1} \;>\; 0
\end{equation*}
for every $1\le i\le \kappa$.

Let $\{\ell_N : N\ge 1\}$ be a sequence of positive integers such that
$1\ll \ell_N\ll N$,
\begin{equation}
\label{08}
 \lim_{N\to\infty} \frac{\ell_N}N  \;=\; 0 \quad 
\text{and}\quad \lim_{N\to\infty}\ell_N\;=\; \infty \;,
\end{equation}
and, for each $b_i$, $1\le i\le \kappa$, define 
$$
\ms E^i_N \;:=\; E_N \cap \Big[\,b_i - \frac{\ell_N}{N}\,,\, 
b_i + \frac{\ell_N}{N}\,\Big].
$$
Since $N^{-1}\ell_N \to 0$, for $N$ large enough $\ms E^i_N \subseteq
V_{b_i}$ for every $1\le i \le \kappa$. In particular, for $N$ large
enough, $\ms E^i_N\cap\ms E^j_N=\varnothing$ for all $i\not =
j$. Moreover, since $\ell_N\to\infty$,
\begin{equation}\label{nuE}
 \lim_{N\to\infty} \nu_N(\Delta_N) \;=\; 0 \quad\text{and} 
\quad \lim_{N\to\infty}\nu_N(\ms E^i_N) \;=\; \frac{m(b_i)}
{\sum_{j=1}^{\kappa} m(b_j)}
\end{equation}
for all $1\le i\le \kappa$, where $\Delta_N := E_N \setminus
(\cup_{i=1}^{\kappa} \ms E^i_N)$.

Fix a positive function $\Phi: [a,b]\to \bb R_+$ bounded above and
below by a strictly positive constant:
\begin{equation*}
0\;<\; \delta \;\le\; \Phi (x) \;\le\; \delta^{-1}\;.
\end{equation*}
This assumption is not needed but we do not seek optimal
assumptions. Consider a birth and death process $\{Z^N_t : t\ge 0\}$
on $E_N$ with rates given by
\begin{equation*}
R_N(x,y) \;=\; 
\left\{
\begin{array}{ll}
{\displaystyle
\Phi(x)} & \text{if $x>y$,}  \\
{\displaystyle
\Phi(y) \, \frac{\nu_N(y)}{\nu_N(x)}} & 
\text{if $x<y$}\;,
\end{array}
\right.
\end{equation*}
provided $x$ and $y$ are neighbors in $E_N$. The process is of course
reversible with respect to the measure $\nu_N$. 

\begin{proposition}
\label{bd}
The Markov process $\{Z^N_t : N\ge 1\}$ exhibits a tunneling behavior
in the sense of \cite[Definition 2.2]{bl2} on the time scale
$N^{1+\alpha}$ with metastates $\ms E^i_N$, metapoints $b_i$, $1\le
i\le \kappa$, and asymptotic Markovian dynamics on
$\{1,\dots,\kappa\}$ characterized by the rates
\begin{eqnarray*}
\!\!\!\!\!\!\!\!\!\!\!\!\!\!\!\!
&& r(i, i+1) \;=\; \frac 1{m(b_i)} \, \frac
1{\int_{b_i}^{b_{i+1}} \{H(u)/\Phi(u)\} \, du}\; , \\
\!\!\!\!\!\!\!\!\!\!\!\!\!\!\!\!
&& \quad r(i+1, i) \;=\; \frac 1{m(b_{i+1})} \, \frac
1{\int_{b_i}^{b_{i+1}} \{H(u)/\Phi (u)\} \, du}\;, \quad 1\le
i<\kappa\;. 
\end{eqnarray*}
\end{proposition}

\begin{proof}
We show that the hypotheses of \cite[Theorem 2.7]{bl2} are in
force. Let $\xi^i_N = b_i$ for $1\le i\le \kappa$. The asymptotic
dynamics has no absorbing point and condition {\bf (H2)} of
\cite[Theorem 2.7]{bl2} follows from \eqref{nuE}. 

To check conditions {\bf (H0)}, {\bf (H1)} we take advantage from the
one-dimensional setting to get explicit expressions for capacities.
For two disjoint subsets $A$, $B$ of $E_N$, denote by $\Cap_N(A,B)$
the capacity between $A$ and $B$. When $A=\{a\}$ we represent
$\Cap_N(A,B)$ by $\Cap_N(a,B)$ with the same convention for $B$.  Let
$x<y$ be points in $E_N$. Recall that
$\Cap_N(x,y)=D_N(f_{x,y})$ where $f_{x,y}:E_N\mapsto \bb
R$ solves the equation $L_Nf_{x,y}(z)=0$ for $z\not \in
\{x,y\}$ with boundary conditions $f_{x,y}(x)=1$ and
$f_{x,y}(y)=0$. An elementary computation gives that $f(z)=1$
for $z\le x$, $f(z)=0$ for $z\ge y$ and
\begin{equation*}
f(z+1/N) - f(z) \;=\; \frac{ \big\{ \nu_N(z) R_N(z,z+1/N) \big\}^{-1}}
{ \sum_{z=x}^{y-1/N} \big\{ \nu_N(z) R_N(z,z+1/N) \big\}^{-1}}\;
\end{equation*}
for $z\in E_N \cap [x,y)$. Hence,
\begin{equation}
\label{cap1}
\Cap_N(x,y) \;=\; \frac{1}{ \sum_{z=x}^{y-1/N} 
\big\{ \nu_N(z) R_N(z,z+1/N) \big\}^{-1}}\;.
\end{equation}
In last two formulae, there is a slight abuse of notation since $E_N$
is not the set $\{z/N : z\in\bb Z \cap [aN, bN]\}$, but the meaning is
clear. In particular, if $\{x^N: N\ge 1\}$, $\{y^N: N\ge 1\}$ are two
sequences in $E_N$ such that $x^N\to a'$ and $y^N\to b'$ for some
$a\le a' < b'\le b$, by (\ref{cap1}),
\begin{equation}
\label{capl}
\lim_{N\to\infty} N^{1+\alpha} \Cap_N(x^N,y^N) \;=\; \frac{1}{\sum_{i=1}^{\kappa}
   m(b_i)}\;\Big\{\;\int_{a'}^{b'} \{H(u)/\Phi (u)\}
\, du \;\Big\}^{-1}\;>\;0\;.
\end{equation}

Denote by $r_N(\ms E^i_N, \ms E^j_N)$ the average jump rate from $\ms
E^i_N$ to $\ms E^j_N$ of the trace of the process $Z^N_t$ on $\ms E_N
= \cup_{1\le m\le \kappa} \ms E^m_N$, defined by \eqref{nv2} in a
general context.  Clearly, $r_N(\ms E^i_N, \ms E^j_N)=0$ for
$|i-j|>1$. Fix an arbitrary $1\le i<\kappa$ and let $G_1=\cup_{j\le i}
\ms E^j_N$, and $G_2=\cup_{j>i}\ms E^j_N$ so that
\begin{eqnarray*}
\nu_N(\ms E^i_N)\,r_N(\ms E^i_N,\ms E^{i+1}_N) &=& \nu_N(b_i+\frac{\ell_N}{N})\,
R^{\ms E}_N(b_i + \frac{\ell_N}{N}, b_{i+1} - \frac{\ell_N}{N}) \\
&=& \nu_N (G_1)\, r_N(G_1,G_2)\;,
\end{eqnarray*}
where $R^{\ms E}_N(x,y)$, $x\not = y \in \ms E_N$, represents the jumps
rates of the trace of $Z^N_t$ on $\ms E_N$.  Therefore, by
\eqref{a03},
\begin{eqnarray*}
r_N(\ms E^i_N,\ms E^{i+1}_N) &=& 
\frac{\nu_N(G_1)\,r_N(G_1,G_2)}{\nu_N(\ms E^i_N)} 
\;=\; \frac{\Cap_N(G_1,G_2)}{\nu_N(\ms E^i_N)}  \\
&=& \frac{\Cap_N(b_i+\ell_N/N,b_{i+1} - \ell_N/N)}{\nu_N(\ms E^i_N)} \;\cdot
\end{eqnarray*}
Analogously, we obtain that
$$
r_N(\ms E^i_N,\ms E^{i-1}_N)\;=\;
\frac{\Cap_N(b_{i}-\ell_N/N , b_{i-1} + \ell_N/N)}
{\nu_N(\ms E^{i}_N)} \;
$$
for any $1<i\le \kappa$. Therefore, by (\ref{capl}) and (\ref{nuE}),
$$
\lim_{N\to\infty} N^{1+\alpha} \, r_N(\ms E^i_N,\ms E^{i+1}_N) 
\;=\; \frac 1{m(b_i)} \, 
\Big\{ \,\int_{b_i}^{b_{i+1}} \{H(u)/\Phi(u)\} \, du \,\Big\}^{-1}\;
$$
and
$$
\lim_{N\to\infty} N^{1+\alpha} \, r_N(\ms E^i_N,\ms E^{i-1}_N) 
\;=\; \frac 1{m(b_{i+1})} \, 
\Big\{\,\int_{b_i}^{b_{i+1}} \{H(u)/\Phi(u)\} \, du \,\Big\}^{-1}\;
$$
for any $1\le i < \kappa$, which concludes the proof of assumption
{\bf (H0)}.

The same arguments show that
\begin{equation}
\label{07}
\begin{split}
& \lim_{N\to\infty} N^{1+\alpha}\, \Cap_N (\ms E^i_N, \breve{\ms E}^i_N) \\
& \qquad \;=\; \frac{1}{\sum_{1\le i\le \kappa} m(b_i)}
\;\Big\{ \frac 1{\int_{b_i}^{b_{i+1}} \{H(u)/\Phi (u)\}\, du}
+ \frac 1{\int_{b_{i-1}}^{b_i} \{H(u)/\Phi (u)\}\, du} \Big\}\;, 
\end{split}
\end{equation}
provided $b_i \not = a$, $b$, with similar identities if $b_i=a$ or if
$b_i=b$.

It remains to check condition {\bf (H1)}.  For any $1\le i \le \kappa$
and $N$ large enough, $H(x)=|x-b_i|^{\alpha}$ for all $x\in{\ms
  E^i_N}$. In consequence, by (\ref{ef04}) and (\ref{cap1}), there
exists a positive constant $C_0$, independent of $N$, such that
\begin{equation*}
\Cap_N(x,b_i) \;\ge\; \frac{C_0}
{\ell_N^{\alpha +1}}\;,
\end{equation*}
for any $1\le i\le\kappa$ and $x\in \ms E^i_N$. Therefore, since
$\ell_N\ll N$, by \eqref{07},
\begin{equation*}
\lim_{N\to\infty} \sup_{x\in \ms E^i_N} \frac  
{\Cap_N (\ms E^i_N, \breve{\ms E}^i_N)}
{\Cap_N(x,b_i)} \;=\; 0
\end{equation*}
for all $1\le i\le\kappa$, which concludes the proof of the
proposition. 
\end{proof}

\section{Potential theory for positive recurrent processes} 
\label{sec2}

We state in this section several properties of continuous time Markov
chains used throughout the article.  Consider a countable set $E$ and
a matrix $R : E\times E \to \bb R$ such that $R(\eta, \xi)\ge 0$,
$\eta\not = \xi$, $-\infty < R(\eta, \eta) <0$, $\sum_{\xi}
R(\eta,\xi)=0$, $\eta\in E$. Let $\lambda(\eta) = -
R(\eta,\eta)$. Since $\lambda(\eta)$ is finite and strictly positive,
we may define the transition probabilities $\{ p(\eta,\xi) :
\eta,\xi\in E \}$ as
\begin{equation}
\label{t01}
p(\eta,\xi) \;=\; \frac 1{\lambda(\eta)} \, R(\eta,\xi)
\quad \textrm{for $\eta\not = \xi$}\;,
\end{equation}
and $p(\eta,\eta)=0$ for $\eta\in E$. We assume throughout this
section that $\{ p(\eta,\xi) : \eta,\xi\in E \}$ are the transition
probabilities of an irreducible and recurrent discrete time Markov
chain, denoted by $Y=\{Y_n : n\ge 0\}$.

Let $\{\eta (t) : t\ge 0\}$ be the unique strong Markov process
associated to the rates $R(\eta,\xi)$. We shall refer to
$R(\cdot,\cdot)$, $\lambda(\cdot)$ and $p(\cdot,\cdot)$ as the
transition rates, holding rates and jump probabilities of $\{ \eta(t) :
t\ge 0\}$, respectively. Since the jump chain $Y$ is irreducible and
recurrent, so is the corresponding Markov process $\{\eta(t) : t\ge
0\}$. We shall assume throughout this section that $\eta(t)$ is
\emph{positive recurrent}.  In consequence, $\eta(t)$ has a unique
invariant probability measure $\mu$. Moreover,
\begin{equation}
\label{m}
M(\eta) \;:=\; \lambda(\eta)\mu(\eta)\;, \quad \eta\in E\;,
\end{equation}
is an invariant measure for the jump chain $Y$, unique up to
scalar multiples. The proofs of these assertions can be found in
Sections 3.4 and 3.5 of \cite{n}. We assume furthermore that the
holding rates are summable with respect to $\mu$:
\begin{equation}
\label{ar}
\sum_{\eta\in E} \lambda(\eta)\, \mu(\eta) \;<\; \infty\;,
\end{equation}
so that $M$ is a finite measure. Assumption (\ref{ar}) reduces the
potential theory of continuous time Markov chains to the potential
theory of discrete time Markov chains.

Let $L^2(\mu)$, $L^2(M)$ be the space of square integrable functions
$f:E\to \bb R$ endowed with the usual scalar product $\<f,g\>_m =
\sum_{\eta\in E} f(\eta) g(\eta) m(\eta)$, with $m=\mu$, $M$,
respectively.  Denote by $P$ the bounded operator in $L^2(M)$ defined
by
\begin{equation}
\label{a01}
(Pf)(\eta) \;=\; \sum_{\xi\in E}  p(\eta,\xi) \, \{f(\xi) -
f(\eta)\} 
\end{equation}
for $f\in L^2(M)$, and by $L$ the generator of the Markov process
$\{\eta(t) : t\ge 0\}$. Thus, for every finitely supported function
$f:E\to \bb R$,
\begin{equation*}
(Lf)(\eta) \;=\; \sum_{\xi\in E}  R(\eta,\xi) \, \{f(\xi) -
f(\eta)\}\;. 
\end{equation*}

Let $L^*$ be the adjoint of the operator $L$ on $L^2(\mu)$. $L^*$ is
the generator of a Markov process $\{\eta^*(t) : t\ge 0\}$ with
holding rates $\lambda^*$ and jump rates $R^*$ given by
$\lambda^*(\eta)=\lambda(\eta)$, $R^*(\eta,\xi) = \mu(\xi) R(\xi,
\eta)/\mu(\eta)$, so that $\mu(\eta) R^*(\eta,\xi) = \mu(\xi) R(\xi,
\eta)$. Let $P^*$ be the adjoint of $P$ in $L^2(M)$. Clearly, the
bounded operator $P^*$ is given by \eqref{a01} with $p^*$ in place of
$p$. Denote by $\{Y^*_n : n\ge 0\}$ the discrete time Markov chain
associated to $P^*$.

Similarly, let $S$ be the symmetric part of the operator $L$ on
$L^2(\mu)$. $S$ is the generator of a Markov process $\{\eta^s(t) :
t\ge 0\}$ with holding rates $\lambda^s$ and jump rates $R^s$ given by
$\lambda^s(\eta)=\lambda(\eta)$, $R^s(\eta,\xi) = (1/2) \{R(\eta,\xi)+
R^*(\eta,\xi)\}$, so that $\mu(\eta) R^s(\eta,\xi) = \mu(\xi) R^s(\xi,
\eta)$.

Let $\mb P_{\eta}$, $\mb P^*_{\eta}$, $\mb P^s_{\eta}$, $\eta\in E$,
be the probability measure on the path space $D(\bb R_+, E)$ of right
continuous trajectories with left limits induced by the Markov process
$\{\eta(t) : t\ge 0\}$, $\{\eta^*(t) : t\ge 0\}$, $\{\eta^s(t) : t\ge
0\}$ starting from $\eta$, respectively. Expectation with respect to
$\mb P_{\eta}$, $\mb P^*_{\eta}$, $\mb P^s_{\eta}$ are denoted by $\mb
E_{\eta}$, $\mb E^*_{\eta}$, $\mb E^s_{\eta}$, respectively.

Denote by $H_A$ (resp. $H^+_A$), $A\subseteq E$, the hitting time of
(resp. return time to) the set $A$:
\begin{equation*}
\begin{split}
& H_{A} \;=\;\inf\{ t> 0 : \eta (t) \in A \}\;, \\
& \quad H^+_A \,=\, \inf\{ t>0 : \eta(t)\in A, \eta (s) \not=\eta (0)
\;\;\textrm{for some $0< s < t$}\}\;.
\end{split}
\end{equation*}

Let $F$ be a proper subset of $E$.  Denote by $\{\mc T_{t} : t\ge 0\}$
the time spent on the set $F$ by the process $\eta(s)$ in the time
interval $[0,t]$:
$$
\mc T_{t} \,:=\, \int_0^t \mb 1\{\eta (s) \in F\} \,ds\;.
$$
Notice that $\mc T_{t}\in \bb R_+$, $\mb P_{\eta}$-a.s.\! for every
$\eta\in E$ and $t\ge 0$. Denote by $\{\mc S_t : t\ge 0\}$ the
generalized inverse of $\mc T_t$:
$$
\mc S_t \,:=\, \sup\{s\ge 0 : \mc T_s \le t \}\,.
$$
Since $\{\eta(t) : t\ge 0\}$ is irreducible and recurrent, $\lim_{t\to
  \infty}\mc T_t = \infty$, $\mb P_{\eta}$-a.s.\! for every $\eta\in
E$. Therefore, the random path $\{\eta^F(t) : t\ge 0\}$, given by
$\eta^F(t) = \eta (\mc S_t)$, is $\mb P_{\eta}$-a.s.\! well defined
for all $\eta\in E$ and takes value in the set $F$.  We call the
process $\{\eta^F(t) : t\ge 0\}$ the trace of $\{\eta(t) : t\ge 0\}$ on
the set $F$.

Denote by $R^F(\eta,\xi)$ the jump rates of the trace process
$\{\eta^F(t) : t\ge 0\}$. By Propositions 6.1 and 6.3 in \cite{bl2},
$\{\eta^F(t) : t\ge 0\}$ is an irreducible, recurrent strong Markov
process whose invariant measure $\mu^F$ is given by
\begin{equation}
\label{g01}
\mu^F(\xi) \;=\; \frac {1}{\mu(F)}\,  
\mu(\xi)\;,\quad \xi\in F\;.
\end{equation}

For each pair $A,B$ of disjoint subsets of $F$, denote by $r_F(A,B)$
the average rate at which the trace process jumps from $A$ to $B$:
\begin{equation*}
r_F(A,B) \;:=\; \frac 1{\mu (A)} \sum_{\eta\in A} \mu (\eta) 
\sum_{\xi \in B} \, R^F(\eta,\xi) \,.
\end{equation*}
By \cite[Proposition 6.1]{bl2},
\begin{equation}
\label{rh}
r_F(A,B)\;=\; \frac 1 {\mu (A)} 
\sum_{\eta\in A} M(\eta) \,{\bf P}_{\eta}\big[ \, H^+_F = H^+_B \, \big]\,.
\end{equation}
We shall refer to $r_F(\cdot,\cdot)$ as the mean set rates associated
to the trace process.

Recall from \cite{gl2} that the capacity between two disjoint subsets
$A$, $B$ of $E$, denoted by $\Cap(A,B)$, is defined as
\begin{equation}
\label{06}
\Cap (A,B) \;:=\; \sum_{\eta\in A} M(\eta) 
\, \mb P_{\eta}\big[ \, H^+_B <H^+_{A} \, \big]\,.
\end{equation}
Hence, by \eqref{rh} for any two disjoint subsets $A$, $B$ of $E$,
\begin{equation}
\label{a03}
\Cap(A,B) \,=\, \mu(A)\, r_{A\cup B} (A,B) \,. 
\end{equation}

Let $\Cap^* (A,B)$, $\Cap^s(A,B)$ be the capacity between two disjoint
subsets $A$, $B$ of $E$ for the adjoint, symmetric process,
respectively. By equation (2.4) and Lemmata 2.3, 2.5 in \cite{gl2},
\begin{equation}
\label{a02}
\Cap^* (A,B)\;=\;\Cap (B,A)\;=\;\Cap (A,B)\;, \quad
\Cap^s(A,B) \;\le\; \Cap (A,B) \;.
\end{equation}

Denote by $\Cap_F$ the capacity with respect to the trace process
$\eta^F(t)$. By the proof of \cite[Lemma 6.9]{bl2}, for every subsets
$A$, $B$ of $F$, $A\cap B = \varnothing$,
\begin{equation}
\label{02}
\mu(F)\,\Cap_F(A,B)\,=\, \Cap(A,B)\,.
\end{equation}

Next result presents an identity between the capacities $\Cap^s(A,B)$
and $\Cap (A,B)$, somehow surprising in view of inequality
\eqref{a02}.

\begin{lemma}
\label{s01}
Let $A$, $B$, $C$ be three disjoint subsets of $E$. Then,
\begin{equation*}
\begin{split}
& \Cap (A, B\cup C) \;+\; \Cap (B, A\cup C) \;-\; \Cap (C, A\cup B) \\
&\qquad \;=\;
\Cap^s (A, B\cup C) \;+\; \Cap^s (B, A\cup C) \;-\; \Cap^s (C, A\cup B)  \;.   
\end{split}
\end{equation*}
\end{lemma}

\begin{proof}
Taking $F=A\cup B \cup C$ in \eqref{02}, we may assume that $A$, $B$,
$C$ forms a partition of $E$. In this case, since $\Cap (C, A\cup B) =
\Cap (A\cup B, C)$ and since $E= A\cup B \cup C$, by \eqref{a03} the
left hand side of the previous equation is equal to
\begin{equation*}
\begin{split}
& \sum_{\eta\in A} \mu(\eta) R(\eta, B\cup C) \;+\;
\sum_{\eta\in B} \mu(\eta) R(\eta, A\cup C) \;-\;
\sum_{\eta\in A\cup B} \mu(\eta) R(\eta, C)  \\
& \qquad
=\; \sum_{\eta\in A} \mu(\eta) R(\eta, B) \;+\;
\sum_{\eta\in B} \mu(\eta) R(\eta, A) 
\; =\; 2 \sum_{\eta\in A} \mu(\eta) R^s(\eta, B)\;,
\end{split}
\end{equation*}
where $R(\xi, D) = \sum_{\zeta\in D} R(\xi,\zeta)$.  Performing the
computations backward with $R^s$ in place of $R$ we conclude the
proof.
\end{proof}

We conclude this section proving a relation between expectations of
time integrals of functions and capacities. Fix two disjoint subsets
$A$, $B$ of $E$. Denote by $f_{AB}$, $f^*_{AB}:E \to\bb R$ the harmonic
functions defined as
\begin{equation*}
f_{AB}(\eta) \;:=\; \mb P_{\eta}\big[\,H_{A} < H_B \,\big]\;, \quad
f^*_{AB}(\eta) \;:=\; \mb P^*_{\eta}\big[\,H_{A} < H_B \,\big] \;.  
\end{equation*}
An elementary computation shows that $f_{AB}$ solves the equation
\begin{equation}
\label{eqf2}
\left\{
\begin{array}{ll}
(L f)(\eta) =0 & \eta\in E\setminus (A\cup B)\;, \\
f(\eta) = 1 & \eta\in A\;, \\
f(\eta) = 0 & \eta\in B \;.
\end{array}
\right.
\end{equation}
and that $f^*_{AB}$ solves the same equation with the adjoint $L^*$
replacing $L$.  Clearly, we may replace the generator $L$ by the
operator $I-P$ in the above equation, and \eqref{eqf2} has a unique
solution in $L^2(M)$ given by $f_{AB}$.

Define the harmonic measure $\nu_{AB}$, $\nu^*_{AB}$ on $A$ as
$$
\nu_{AB}(\eta)\,=\, \frac{M(\eta) \,
\mb P_{\eta} \big[ \, H^+_B <H^+_{A} \, \big] }{\Cap(A,B)}\,,
\quad \nu^*_{AB}(\eta)\,=\, \frac{M(\eta) \,
\mb P^*_{\eta} \big[ \, H^+_B <H^+_{A} \, \big] }{\Cap^*(A,B)}
\quad \eta\in A\,.
$$
Denote by ${\bf E}_{\nu_{AB}}$ the expectation associated to the
Markov process $\{\eta(t) : t\ge 0\}$ with initial distribution
$\nu_{AB}$.

\begin{proposition}
\label{bovier}
Fix two disjoint subsets $A$, $B$ of $E$.  Let $g:E \to\bb R$ be a
$\mu$-integrable function. Then,
\begin{equation}
\label{01}
\mb E_{\nu^*_{AB}}\Big[ \int_0^{H_B} g(\eta(t))\,dt \Big] 
\;=\; \frac{\langle\, g \,,\, f^*_{AB}\rangle_{\mu}\, }
{\Cap(A,B)}\;,  
\end{equation}
where $\<\cdot , \cdot \>_\mu$ represents the scalar product in
$L^2(\mu)$.
\end{proposition}

\begin{proof}
We first claim that the proposition holds for indicator functions of
states.  Fix an arbitrary state $\xi\in E$. If $\xi$ belongs to $B$
the right hand the left hand side of \eqref{01} vanish. We may therefore
assume that  $\xi$ does not belong to  $B$. In this case we may write
the expectation appearing in the statement of the lemma as
\begin{equation*}
{\bf E}_{\nu^*_{AB}}\Big[\, 
\sum_{n=0}^{\bb H_B -1} \frac{e_n}{\lambda(\xi)} \,{\bf 1}\{ Y_n=\xi \} \,
\Big] \;, 
\end{equation*}
where $\{Y_n : n\ge 0\}$ is the discrete time embedded Markov chain,
$\{e_n : n\ge 0\}$ is a sequence of i.i.d.\! mean one exponential
random variables independent of the jump chain $\{Y_n : n\ge 0\}$, and
$\bb H_B$ the hitting time of the set $B$ for the discrete time Markov
chain $Y_n$. By the Markov property and by definition of the harmonic
measure $\nu^*_{AB}$, this expression is equal to
\begin{equation*}
\begin{split}
& \frac{1}{\lambda(\xi)} \sum_{\eta\in A} \sum_{n\ge 0} \nu^*_{AB} (\eta) \,  
{\bf P}_{\eta} \big[ Y_n=\xi \,,\, n< \bb H_B\Big] \\
& \qquad =\; \frac{1}{\lambda(\xi)\, \Cap^*(A,B)} \sum_{n\ge 0}
\sum_{\eta\in A} M(\eta)\, {\bf P}^*_{\eta} \big[\, H_B < H^+_A  \,\big]
{\bf P}_{\eta} \big[ Y_n=\xi \,,\, n< \bb H_B\big]\;.
\end{split}
\end{equation*} 
We may replace the hitting time and the return time $H_B$, $H^+_A$ by
the respective times $\bb H_B$, $\bb H^+_A$ for the discrete chain. On
the other hand, since $\eta$ and $\xi$ do not belong to $B$, the event
$\{Y_0=\eta \,,\, Y_n=\xi \,,\, n< \bb H_B\}$ represents all paths
that started from $\eta$, reached $\xi$ at time $n$ without passing
through $B$. In particular, by the detailed balanced relations between
the process and its adjoint, $M(\eta)\, {\bf P}_{\eta} [ Y_n=\xi \,,\,
n< \bb H_B] = M(\xi)\, {\bf P}^*_{\xi} [ Y_n=\eta \,,\, n< \bb H_B]$
and the last sum becomes
\begin{equation*}
\begin{split}
& \frac{M(\xi)}{\lambda(\xi)\, \Cap^*(A,B)} \sum_{n\ge 0}
\sum_{\eta\in A}  {\bf P}^*_{\eta} \big[\, \bb H_B < \bb H^+_A  \,\big]
{\bf P}^*_{\xi} \big[ Y_n=\eta \,,\, n< \bb H_B\big] \\
&\qquad
=\; \frac{M(\xi)}{\lambda(\xi)\, \Cap^*(A,B)} \sum_{n\ge 0}
\sum_{\eta\in A}
{\bf P}^*_{\xi} \big[ Y_n=\eta \,,\, n< \bb H_B 
\;,\, \bb H_B \circ \theta_n < \bb H^+_A \circ \theta_n \big] \;,
\end{split}
\end{equation*}
where we used the Markov property in the last step. In this formula
$\{\theta_k : k\ge 1\}$ stands for the group of discrete time shift.
Summing over $\eta$ the sum can be written as
\begin{equation*}
\frac{M(\xi)}{\lambda(\xi)\, \Cap^*(A,B)} \sum_{n\ge 0}
{\bf P}^*_{\xi} \big[ Y_n \in A \,,\, n< \bb H_B 
\;,\, \bb H_B \circ \theta_n < \bb H^+_A \circ \theta_n \big]
\end{equation*}
The set inside the probability represents the event that the process
$Y_k$ visits $A$ before visiting $B$ and that its last visit to
$A$ before reaching $B$ occurs at time $n$. Hence, since $M(\xi) =
\lambda(\xi) \mu(\xi)$, since by \eqref{a02} $\Cap^*(A,B) = \Cap(A,B)$
and since $g$ is the indicator of the state $\xi$, summing over $n$ we
get that the previous expression is equal to
\begin{equation*}
\frac 1{\Cap^*(A,B)} \,
\mu (\xi)\, {\bf P}^*_{\xi} \big[ \bb H_A < \bb H_B  \big]
\;=\; \frac{\langle\, g \,,\, f^*_{AB}\rangle_{\mu}}
{\Cap(A,B)} \;\cdot
\end{equation*}

By linearity and the monotone convergence theorem we get
the desired result for positive and then $\mu$-integrable functions.
\end{proof}

In the particular case where $A=\{\eta\}$ for $\eta\not\in B$ we have
that
\begin{equation}
\label{g02}
\mb E_{\eta} \Big[\, \int_0^{H_B} g(\eta (s))\,ds \, \Big] \;=\; 
\frac{ \langle \,g\,,\, f^*_{\{\eta\} B} \, \rangle_{\mu}}
{\Cap(\{\eta\},B)}
\end{equation}
for any $\mu$-integrable function $g$.  This formula provides an
estimation for the thermalization. 

Let $S$ be a finite set, let $\pi=\{A^x : x\in S\}$ be a partition of
$E$, and let $\xi_x$ be a state in $A^x$ for each $x\in S$. For each
$\mu$-integrable function $g$ denote by $\langle
g|\pi\rangle_{\mu}:E\to\bb R$ the conditional expectation of $g$,
under $\mu$, given the $\sigma$-algebra generated by $\pi$:
$$
\langle g|\pi\rangle_{\mu}\,=\,\sum_{x\in S} \frac{\langle g \, \mb
1\{A^x\}\, \rangle_{\mu}}{\mu(A^x)}\, \mb 1 \{A^x\}\,. 
$$
For each $x\in S$, let
$$
\Cap(\xi_x)\,:=\,\inf_{\eta\in A^x\setminus\{\xi_x\}}
\Cap(\{\eta\}, \{\xi_x\})\,.
$$
The next result shows that if the process thermalizes quickly in each
set of the partition, we may replace time averages of a bounded
function by time averages of the conditional expectation. This
statement plays a key role in the investigation of metastability. It
assumes, however, the existence of an attractor.

\begin{corollary}
\label{s09}
Let $g:E\to\bb R$ be a $\mu$-integrable function. Then, for every
$t>0$,
\begin{equation*}
\sup_{\eta\in E} \Big\vert \mb E_{\eta} \Big[ \int_0^t \big\{
g-\langle g | \pi\rangle_{\mu}\big\} (\eta (s)) \, ds \Big] \Big\vert 
\;\le\; 4 \sum_{x\in S}
\frac{ \langle \,|g|\, \mb 1\{A^x\}\, \rangle_{\mu}}{\Cap(\xi_x)} \;, 
\end{equation*}
where $|g|(\eta)=|g(\eta)|$ for all $\eta$ in $E$.
\end{corollary}

The proof of this result follows from \cite[Corollary 6.5]{bl2},
formula (\ref{g02}) and the fact that $f^*_{AB}$ is bounded by one.


\begin{thebibliography}{99}
\bibitem{bl2} J. Beltr\'an, C. Landim: Tunneling and metastability of
  continuous time Markov chains. J. Stat. Phys.  {\bf 140} 1065--1114
  (2010).

\bibitem{bl3} J. Beltr\'an, C. Landim: Metastability of reversible
  condensed zero range processes on a finite
  set. Probab. Th. Rel. Fields {\bf 152} 781--807 (2012)

\bibitem{bl4} J. Beltr\'an, C. Landim; Metastability of reversible finite
  state Markov processes. Stoch. Proc. Appl. {\bf 121} 1633--1677 (2011).

\bibitem{bl5} J. Beltr\'an, C. Landim; Tunneling of the Kawasaki dynamics at
  low temperatures in two dimensions. arXiv:1109.2776 (2011). 

\bibitem{begk1} A. Bovier, M. Eckhoff, V. Gayrard, M. Klein.
  Metastability in stochastic dynamics of disordered mean field
  models. Probab. Theory Relat. Fields {\bf 119}, 99-161 (2001)

\bibitem{begk2} A. Bovier, M. Eckhoff, V. Gayrard, M. Klein.
  Metastability and low lying spectra in reversible Markov chains.
  Commun. Math. Phys. {\bf 228}, 219--255 (2002).

\bibitem{cgov}  M. Cassandro, A. Galves, E. Olivieri, M. E. Vares.
  Metastable behavior of stochastic dynamics: A pathwise approach.
  J. Stat. Phys. {\bf 35}, 603--634 (1984).

\bibitem{c} K. L. Chung, {\em Markov chains with stationary transition
    probabilities}. Second edition. Die Grundlehren der mathematischen
  Wissenschaften, Band 104 Springer-Verlag New York, Inc., New York
  1967.

\bibitem{f} David Freedman. \emph{Markov chains}. Holden-Day, San
  Francisco (1971).

\bibitem{g} A. Gaudilli\`{e}re. Condenser physics applied to Markov
  chains: A brief introduction to potential theory. Online available
  at http://arxiv.org/abs/0901.3053.

\bibitem{gl2} A. Gaudilli\`ere, C. Landim: A Dirichlet principle for
  non reversible Markov chains and some recurrence
  theorems. arXiv:1111.2445 (2011).

\bibitem{jlt1} M. Jara, C. Landim, A. Teixeira; Quenched scaling limits of
  trap models. Ann. Probab. {\bf 39}, 176--223 (2011).

\bibitem{jlt2} M. Jara, C. Landim, A. Teixeira; Quenched scaling
  limits of trap models in random graphs. preprint (2012).

\bibitem{gl3} C. Landim: Metastability for a non-reversible dynamics:
  the evolution of the condensate in totally asymmetric zero range
  processes. arXiv:1204.5987

\bibitem{lp} J. L. Lebowitz, O. Penrose: Rigorous treatment of the van
  der Waals--Maxwell theory of the liquid--vapor transition.
  J. Math. Phys. {\bf 7}, 98--113 (1966).

\bibitem{n} J. R. Norris. {\em Markov chains}. Cambridge University
  Press, Cambridge (1997).

\bibitem{ov} E. Olivieri and M. E. Vares. {\em Large deviations and
    metastability}. Encyclopedia of Mathematics and its Applications,
  vol. 100. Cambridge University Press, Cambridge, 2005.
\end{thebibliography}
\end{document}